\documentclass{amsart}
\usepackage{amssymb}
\usepackage{amsmath}
\usepackage{latexsym}
\usepackage{amsthm}
\newtheorem{prop}{Proposition}[section]
\newtheorem{thm}[prop]{Theorem}

\newtheorem{lemma}[prop]{Lemma}

\newtheorem{ex}[prop]{Example}

\newtheorem{defn}[prop]{Definition}
\newtheorem{rem}[prop]{Remark}
\begin{document}
\title{Elasticity in Polynomial-Type Extensions}
\author{Mark Batell}
\address{Department of Mathematics\\
	North Dakota State University\\
	Fargo, ND 58108}
\email{mark.batell@ndsu.edu}
\author{Jim Coykendall}
\address{Department of Mathematics\\
	North Dakota State University\\
	Fargo, ND 58108}
\email{jim.coykendall@ndsu.edu}

\begin{abstract}
The elasticity of an atomic integral domain is, in some sense, a measure of how far the domain is from being a unique factorization domain (or, more properly, a half-factorial domain). We consider the relationship between the elasticity of a domain, $R$, and the elasticity of its polynomial ring $R[x]$. For example, if $R$ has at least one atom, a sufficient condition for the polynomial ring $R[x]$ to have elasticity $1$ is that every nonconstant irreducible polynomial $f \in R[x]$ be irreducible in $K[x]$. We will determine the integral domains $R$ whose polynomial rings satisfy this condition.

\end{abstract}
\maketitle

\section{Introduction and Motivation}

In this paper, $R$ will always be an integral domain with quotient field $K$. The notation $\text{Irr}(R)$ will stand for the irreducible elements of $R$, $\text{A}(R)$ will be the elements of $R$ that can be expressed as a product of atoms, $U(R)$ and $Cl(R)$ will respectively denote the unit group and the class group of $R$. For a nonzero nonunit element $x\in\text{A}(R)$, we define the elasticity of $x$ to be 

\[
\rho(x)=\text{sup}\{\frac{n}{m}\vert x=\pi_1\pi_2\cdots\pi_n=\xi_1\xi_2\cdots\xi_m\}
\]

\noindent where each $\pi_i,\xi_j$ is an irreducible element of $R$. For example, if $x$ is a product of primes or is an element of a half-factorial domain, then $\rho(x)=1$. If $x\notin\text{A}(R)$ then $\rho(x)$ is undefined.

For an integral domain $R$, the elasticity is defined by

\[
\rho(R)=\text{sup}\{\rho(x)\vert x\in\text{A}(R)\}.
\]

As previously, we say that the elasticity of a domain without any atoms is undefined. It is well-known that if $R$ is atomic, then $\rho(R)=1$ if and only if $R$ is an HFD, but in the nonatomic case, the situation can be more exotic. For example in \cite{CZ} a domain was constructed with a unique (up to associates) irreducible element. Such a domain, $R$, is necessarily nonatomic, but $\rho(R)=1$.

More generally in [CM] it is shown that any atomic monoid can be realized as the ``atomic part" of an integral domain (again, usually non-atomic). Hence, one can construct nonatomic domains that display any prescribed elasticity.

Given an integral domain, $R$, it is natural to ask what is the relationship between $\rho(R)$ and $\rho(R[x])$. Since any factorization of a constant in $R[x]$ must be a factorization in $R$, it is easy to see that, in general, $\rho(R[x])\geq \rho(R)$. More generally, one can ask if we have the sequence of integral domains

\[
R=R_0\subseteq R_1\subseteq R_2\subseteq\cdots\subseteq K
\]

\noindent what is the relationship between $\rho(R_0+R_1x+R_2x^2+\cdots)$ and the collection of data $\rho(R_i)$? Some special cases of this general construction worth noting are the polynomial ring ($R_i=R$ for all $i\geq 0$), the construction $R+xK[x]$ ($R_0=R$ and $R_i=K$ for all $i\geq 1$), and $R+Rx+x^2K[x]$ ($R_0=R_1=R$ and $R_i=K$ for all $i\geq 2$).

\section{Some Polynomial-Type Constructions}

We begin this section with some preliminary lemmata, but first we recall the notion of an AP-domain.

\begin{defn}
We say that the integral domain, $R$, is an AP-domain if every irreducible (atom) in $R$ is prime.
\end{defn}

\begin{lemma}
\label{AP}
Let $R$ be an AP-domain with at least one irreducible element, then $\rho(R)=1$.
\end{lemma}

\begin{proof}
Of course, if $R$ is an AP-domain vacuously (that is, in the case that $R$ has no atoms), then $\rho(R)$ is undefined. Suppose, on the other hand, that $\text{Irr}(R)$ is nonempty. Since all atoms in an AP-domain, are prime, any irreducible factorization is a prime factorization and therefore is unique. Hence $\rho(x)=1$ for all $x\in\text{A}(R)$ and so $\rho(R)=1$. 
\end{proof}

\begin{lemma}
\label{shrink}
Let $R$ be a domain, $p\in R$ be a nonzero prime element, and $a\in\text{A}(R)$. Then $\rho(a)\geq \rho(ap)$.
\end{lemma}

\begin{proof}
Suppose we have the following irreducible factorization of $ap$:

\[
ap=\xi_1\xi_2\cdots\xi_n,
\]

\noindent where each $\xi_i\in\text{Irr}(R)$ for all $1\leq i\leq n$. Since $p\in R$ is prime, $p$ must be associated with one of the $\xi_i$; we will say, without loss of generality, that $\xi_n=up$ for some $u\in U(R)$. Since $R$ is an integral domain, we cancel the factor of $p$ to obtain

\[
a=\xi_1\xi_2\cdots\xi_{n-1}u.
\]

The upshot is that an arbitrary irreducible factorization of $ap$ is (up to associates) $p$ times an irreducible factorization of $a$. Hence there is a factorization of $ap$ of length $m+1$, if and only if there is a corresponding factorization of $a$ of length $m$. Since $m\geq k\geq 1$ implies that $\frac{m}{k}\geq\frac{m+1}{k+1}$, we have that $\rho(a)\geq\rho(ap)$. 
 
\end{proof}

\begin{rem}
To tie up a loose end, we note that the inequality in the previous result can be strict. For example, in the ring $\mathbb{Z}[\sqrt{-14}]$ the element 81 has precisely two irreducible factorizations (up to associates and reordering):

\[
81=(3)(3)(3)(3)=(5+2\sqrt{-14})(5-2\sqrt{-14}).
\]

So $\rho(81)=2$. Now consider the element 81 as an element of $\mathbb{Z}[\sqrt{-14}][x]$. As before 81 has only two irreducible factorizations (the ones mentioned previously) and $x$ is a prime element. An easy check shows that $\rho(81x)=\frac{5}{3}<2$.

\end{rem}

We now present the following theorem.

\begin{thm}
Let $R$ be an integral domain with quotient field $K$. If $R$ contains at least one atom, then $\rho(R)=\rho(R+xK[x])$. If $R$ has no atoms (that is, $R$ is an antimatter domain) then $\rho(R)$ is undefined and $\rho(R+xK[x])=1$.
\end{thm}

\begin{proof}
Let $g(x)$ be a nonconstant polynomial in $R+xK[x]$. We claim that if $g(x)$ is irreducible, then $g(x)$ is (up to associates) either $x$ or of the form $1+xf(x)$ where $1+xf(x)\in\text{Irr}(K[x])$. To see this, note that if $g(x)$ is nonconstant, then $g(x)=r+xk(x)$ with $k(x)\in K[x]\setminus\{0\}$. If $r=0$ then the stipulation that $g(x)$ is irreducible forces the condition $k(x)\in U(R)$. On the other hand, if $r\neq 0$, then the factorization $g(x)=r(1+\frac{1}{r}xk(x))$ shows that if $g(x)$ is irreducible, then $r\in U(R)$ and $1+\frac{1}{r}xk(x)\in\text{Irr}(K[x])$. This establishes the claim.

We also note that the elements $x$ and $1+xf(x)\in\text{Irr}(K[x])$ are, in fact, prime elements of $R+xK[x]$. The fact that $x$ is prime is straightforward. For an irreducible of the form $1+xf(x)$, note that if $1+xf(x)$ divides the product $h(x)k(x)$ (with $h(x), k(x)\in R+xK[x]$) then without loss of generality, $1+xf(x)$ divides $h(x)$ in $K[x]$. We say that $h(x)=(1+xf(x))q(x)$, and comparing constant terms, we see that $q(x)\in R+xK[x]$. Hence $x$ and irreducibles of the form $1+xf(x)$ are prime in $R+xK[x]$.

From the previous observations, we see that if $R$ is an antimatter domain, then $R+xK[x]$ is an AP-domain (with $\text{Irr}(R+xK[x])$ nonempty) and hence $\rho(R+xK[x])=1$ by Lemma \ref{AP}. 

Now suppose that $R$ has at least one irreducible element. Since any element of $R$, factored as an element of $R[x]$, has only factors from $R$ (and any irreducible in $R$ remains irreducible in $R[x]$), we have that $\rho(R+xK[x])\geq\rho(R)$. On the other hand, let $f(x)\in\text{A}(R+xK[x])$. We factor $f(x)$ into irreducibles as follows:

\[
f(x)=\pi_1\pi_2\cdots\pi_m g_1(x)g_2(x)\cdots g_k(x)
\]

\noindent with $\pi_i\in\text{Irr}(R)$ and $g_i(x)\in\text{Irr}(R+xK[x])$ of degree at least 1. By our previous remarks, each $g_i(x)$ is prime. Hence by Lemma \ref{shrink}, $\rho(f(x))\leq\rho(\pi_1\pi_2\cdots\pi_m)$. Hence $\rho(R+xK[x])\leq\rho(R)$ and so, we have equality.

\end{proof}

In stark contrast, the next result shows that a minor tweaking of the previous construction can yield a domain with infinite elasticity. This also gives a strong indication of the level of difficulty of determining the elasticity of $R_0+R_1x+R_2x^2+\cdots$ in terms of the elasticities $\rho(R_i)$.

\begin{prop}
Let $R$ be a domain that contains at least one atom, then 

\[
\rho(R+Rx+x^2K[x])=\infty.
\]
\end{prop}

\begin{proof}
Let $\pi\in\text{Irr}(R)$. For all $n\in\mathbb{N}_0$ the polynomial $(\pi^n\pm x)\in\text{Irr}(R+Rx+x^2K[x])$. The irreducible factorizations

\[
(\pi^n+x)(\pi^n-x)=\pi^{2n}(1-\frac{1}{\pi^{2n}}x^2)
\]

\noindent have lengths $2$ and $2n+1$ respectively. Hence we see that $\rho(R+Rx+x^2K[x])=\infty$.
\end{proof}

We now specialize to the case $R[x]$. In comparing the elasticities $\rho(R)$ and $\rho(R[x])$, there are two dynamics to consider. The first is the factorization of constants (which is reflected in $\rho(R)$) and the different factorizations that may result from the polynomial structure. To illustrate we consider the following examples.

\begin{ex}
It is well-known (see for example \cite{C1}) that $\mathbb{Z}[\sqrt{-3}]$ is a half-factorial domain (and hence has elasticity 1). The domain $\mathbb{Z}[\sqrt{-3}][x]$ is not an HFD. The irreducible factorizations

\[
(2x+(1+\sqrt{-3}))(2x+(1-\sqrt{-3}))=(2)(2)(x^2+x+1)
\]

\noindent demonstrates that the elasticity of the polynomial extension exceeds 1. 
\end{ex}

A close look at the mechanics of the previous example reveals that the failure of the domain $\mathbb{Z}[\sqrt{-3}]$ to be integrally closed allowed the creation of this ``bad factorization." In the proof of the main theorem in \cite{C2} it is shown that if $R$ is not integrally closed, a simliar effect occurs.

It is known (see \cite{Za}) that if $R$ is a Krull domain, then $R[x]$ is an HFD if and only if $\vert Cl(R)\vert\leq 2$. It is also known from \cite{Ca} that if $R$ is a ring of algebraic integers (and hence, certainly a Krull domain), then $R$ is an HFD 
if and only if $\vert Cl(R)\vert\leq 2$. Hence if $R$ is a ring of algebraic integers with $\vert Cl(R)\vert\leq 2$, then $R$ is an HFD and so is $R[x]$. In this case $\rho(R)=\rho(R[x])$, but the equality can be delicate as we will demonstrate in the following example. The following example should be constrasted with the previous as it is integrally closed.

\begin{ex}
\label{integrally closed}
The integral domain $R:=\mathbb{Z}[\sqrt{-5}]$ is a ring of integers of class number precisely 2 (see, for example, the tables in \cite{Co}) and hence is an HFD that does not have unique factorization. But although $\rho(R[x])=1$, the factorizations can be exotic. The elements $2x^2+2x+3$, $2$, $2x+1+\sqrt{-5}$, and $2x+1-\sqrt{-5}$ are all elements of $\text{Irr}(R[x])$. Consider the factorizations

\[
(2)(2x^2+2x+3)=(2x+1+\sqrt{-5})(2x+1-\sqrt{-5}).
\]

The upshot is that even in this relatively ``nice" domain, the factorizations of elements can depend on how the polynomials break down (with respect to degree) in a nontrivial way. 
\end{ex}

It is well-known that if $R$ is a UFD with quotient field $K$, then any irreducible polynomial over $R[x]$ remains irreducible over $K[x]$. More generally, domains, $R$, for which every irreducible polynomial of degree at least one remains irreducible in $K[x]$ would seem to be the basic case to solve. For these domains, it would seem likely that there is a more direct correlation between $\rho(R)$ and $\rho(R[x])$, since there must be a one to one correspondence between irreducible factors of degree at least 1 for any two irreducible factorizations of the same element. Certainly, bad factorizations of the ilk of the previous two examples are precluded (although we feel obligated to point out again that the second example is an HFD).

Although it may seem reasonable to consider domains where irreducibles of degree at least one in $R[x]$ remain irreducible in $K[x]$, it is not obvious that in this case $\rho(R)=\rho(R[x])$. To illustrate the problem, consider the irreducible factorizations

\[
\pi_1\pi_2\cdots\pi_kf_1(x)f_2(x)\cdots f_m(x)=\xi_1\xi_2\cdots\xi_tg_1(x)g_2(x)\cdots g_n(x)
\]

\noindent with each $\pi_i, \xi_j\in\text{Irr}(R)$ and $f_i(x), g_j(x)$ all irreducible of degree at least one. 

Even if we have that $m=n$ and each $f_i(x)$ and $g_j(x)$ pair off (up to units in $K$), there is no guarantee that the ratio of $k$ and $t$ are within the elasticity bounds of $R$ (precisely because there is ambiguity up to units in $K$).

That being said, we present the following theorem. The rest of the paper will be devoted to establishing this result. The theorem will follow quickly from our classification of domains with this property.

\begin{thm}
\label{semi}
Let $R$ be a domain such that every irreducible of $R[x]$ of degree greater than or equal to 1 is irreducible in $K[x]$. Then if $\rho(R)$ is defined, then $\rho(R)=\rho(R[x])$.
\end{thm}

As noted before, these conditions are not necessary as Example \ref{integrally closed} shows.

\section{Irreducibles of $R[x]$ versus irreducibles of $K[x]$}

In this section, we use the following facts and definitions many times without further mention.

\begin{itemize}
\item[a)] If $F$ is a nonzero fractional ideal of $R$, then $F_v=(F^{-1})^{-1}$.
\item[b)] The elements $a_1,a_2, \cdots ,a_n \in R$ have a greatest common divisor provided that $(a_1,a_2, \cdots ,a_n)_v$ is principal.
\item[c)] If $f \in K[x]$, the ideal generated by the coefficients of $f$ is denoted $A_f$.
\item[d)] The greatest common divisor of $a_1,a_2, \cdots ,a_n \in R$ will be denoted $[a_1,a_2, \cdots ,a_n]$ if it exists.
\end{itemize}

Our goal in this last section is to characterize those domains $R$ having the following property:

\begin{itemize}
\item[(P)] every nonconstant irreducible polynomial $f \in R[x]$ is irreducible in $K[x]$
\end{itemize}

The techniques used in the proof we shall give strongly resemble those used in the classical proof that every polynomial ring $R[x]$ over a UFD is again a UFD.  The proof of this classical result boils down to showing that any UFD, $R$, satisfies property (P) above.  Gauss's Lemma, which states that the product of two primitive polynomials is primitive, is the key which allows this proof to go through in the UFD case.  

Conditions under which the product of two primitive polynomials remains primitive has been studied in more general domains (see for instance, \cite{T}).  It turns out that the domains satisfying property (P) must satisfy a condition somewhat stronger than Gauss's Lemma; they must satisfy what is called the \emph{PSP-property}.  

\begin{defn}
A domain $R$ has the PSP-property if whenever $a_0+a_1x+\cdots+a_nx^n$ is a primitive polynomial over $R$ and $z \in (a_0,a_1,\cdots,a_n)^{-1}$, then $z \in R$.
\end{defn}

Polynomials $a_0+a_1x+\cdots+a_nx^n$ satisfying the above definition are called \emph{superprimitive}. Thus a domain has the PSP-property if every primitive polynomial is superprimitive.  

For integral domains, the following implications are well-known

\[
\text{UFD}\Longrightarrow\text{GCD}\Longrightarrow\text{PSP-property}\Longrightarrow\text{GL-property}\Longrightarrow\text{AP-property}.
\]  

\noindent and in \cite{AQ} it is shown that all of these types are equivalent for atomic domains.

Arnold and Sheldon \cite[Example 2.5]{AS} gave an example of a domain satisfying Gauss's Lemma (such domains are said to have the \emph{GL-property}), but failing to have the PSP-property.  The domain they considered was the domain $F[\{ x^\alpha : \alpha \geqslant 0 \},\{ y^\alpha : \alpha \geqslant 0 \}, \{ z^\alpha x^\beta : \alpha, \beta > 0 \}, \{ z^\alpha y^\beta : \alpha, \beta > 0 \}]$. Here, all exponents $\alpha$ and $\beta$ are understood to be taken from the field $\mathbb{Q}$ of rational numbers. This is an example of a monoid domain, and can intuitively be thought of as the ring of all formal polynomials in the given "indeterminates" with coefficients in $F$, the field of two elements.  We note that $yt+x$ is a primitive polynomial in $t$ that fails to be superprimitive, as $z \in (x,y)^{-1}$.  This leads us to the following theorem.

\begin{prop}
Assume every nonconstant irreducible $f \in R[x]$ is irreducible in $K[x]$.  Then $R$ is integrally closed and has the PSP-property.
\end{prop}

\begin{proof}
Assume $R$ is not integrally closed.  Choose an element $\omega \in K$ that satisfies a monic irreducible polynomial $f \in R[x]$ of degree $\geq 2$.  Since $\omega$ is a root of $f$, the division algorithm in $K[x]$ implies that $f=(x-\omega)g$, where $g$ is a polynomial in $K[x]$ of degree $\geq 1$.  Hence $f$ is irreducible over $R$ but reducible over $K$. This is a contradiction.

Next, we assume $R$ does not have the PSP-property.  Let $y_0+y_1x+\cdots+y_nx^n$ be a primitive polynomial and let $z \in K$ be such that $z \in (y_0,y_1, \cdots, y_n)^{-1}$ but $z \notin R$.  In the collection of all primitive polynomials that are not superprimitive, we assume that $I:=y_0+y_1x+\cdots+y_nx^n$ is one of minimal degree.  In $K[x]$ we have the factorization $y_nx^{n+1}+(y_{n-1}+zy_n)x^n+\cdots+(y_0+zy_1)x+zy_0=(x+z)(y_nx^n+y_{n-1}x^{n-1}+\cdots+ y_1x+y_0)$ where the polynomial $f$ on the left hand side belongs to $R[x]$.  We claim that $f$ is irreducible over $R$.  If $f=gh$ for some $g,h \in R[x]$ then $x+z$ divides $g$ or $h$ in $K[x]$, say $g=(x+z)p(x)$. Since $R$ is integrally closed, $p(x) \in R[x]$ (see \cite[Theorem 10.4]{G}), say   $p(x)=a_kx^k+a_{k-1}x^{k-1}+ \cdots +a_1x+a_0$. Then $g(x)=a_kx^{k+1}+(a_{k-1}+za_k)x^k+\cdots+(a_0+za_1)x+za_0$, so that $za_0, za_1,\cdots,za_k \in R$, and hence $p(x)$ is not superprimitive.  But $p(x)$ is a factor of the primitive polynomial $I$.  Hence $p(x)$ is primitive, so that the minimality assumption on $I$ implies that $k=n$.  It follows that $h$ is a unit so that $f$ is irreducible over $R$, but not over $K$, the desired contradiction. 
\end{proof}

Thus in our search for the domains satisfying property (P), we may restrict our attention to integrally closed domains having the PSP-property.  

To show that a particular domain actually has property (P), one would probably employ the following strategy: Suppose $f \in R[x]$ is a nonconstant polynomial that is irreducible over $R$, but fails to be irreducible over the field of fractions $K$, say $f=gh$ in $K[x]$.  Now ``clear the denominators," that is, choose nonzero $b,d \in R$ such that $bdf=(bg)(dh)$ and $bg,dh \in R[x]$.  At this point in the proof, we would probably need to find some way to cancel out $b$ and $d$ to get a contradiction, namely, that $f=g^\prime h^\prime$ for some $g^\prime, h^\prime \in R[x]$.  It turns out that if $R$ has the PSP-property, then we can assert, after clearing denominators, that the greatest common divisor of the coefficients of $bdf$ exists and is equal to $bd$, as shown by the following.

\begin{prop}
\label{gcd}
Let $R$ be a domain. The following are equivalent.
\begin{itemize}
\item[a)] $R$ has the PSP-property.
\item[b)] Whenever the elements $a_1,a_2,\cdots,a_n \in R$ have a greatest common divisor and $0\neq b \in R$, then  $[ba_1,ba_2,\cdots,ba_n]=b[a_1,a_2,\cdots,a_n]$.  
\end{itemize}
\end{prop}

\begin{proof}
Assume $R$ has the PSP-property and $[a_1,a_2,\cdots,a_n]=g$.  Given $b \in R$, it is clear that $bg$ is a common divisor of $ba_1,ba_2,\cdots,ba_n$. If $x$ is some other common divisor, then $\frac{bg}{x} \in (\frac{a_1}{g},\frac{a_2}{g},\cdots,\frac{a_n}{g})^{-1}$. This implies  $\frac{bg}{x} \in R$ since $R$ is PSP. In other words, $x|bg$ so that $[ba_1,ba_2,\cdots,ba_n]=bg$

Conversely, assume b) holds and $\frac{r}{s} \in (a_1,a_2,\cdots,a_n)^{-1}$, where $a_0+a_1x+\cdots+a_nx^n$ is some primitive polynomial over $R$.  Then $s|ra_i$ for all $i$, so b) implies that $s|r$. Thus $\frac{r}{s} \in R$, so that $R$ is PSP. 
\end{proof}

\begin{prop}
Let $R$ be an integrally closed PSP-domain. The following are equivalent.
\begin{itemize}
\item[a)] Every nonconstant irreducible polynomial $f \in R[x]$ is irreducible in $K[x]$
\item[b)] Whenever $f=gh$ in $R[x]$ and the greatest common divisor of the coefficients of $f$ exists, then the greastest common divisor of the coefficients of $g$ exists
\end{itemize}
\end{prop}

\begin{proof}

b) $\Longrightarrow$ a).  Assume b) holds.  Let $f \in R[x]$ be a nonconstant irreducible polynomial (hence the greatest common divisor of the coefficients is 1).  Suppose $f=gh$, where $g,h \in K[x]$ have degrees $\geqslant 1$.  Choose nonzero $b,d \in R$ such that $bg,dh \in R[x]$.  Then we have the equation $bdf=(bg)(dh)$ in $R[x]$, and since $R$ is PSP, Proposition \ref{gcd} implies that the greatest common divisor of the coefficients of $bdf$ exists and is equal to $bd$.  Hence the greatest common divisor of the coefficients of $bg$ exists, say $u$, and the greatest common divisor of the coefficients of $dh$ exists, say $v$.  Note that $uv$ divides the coefficients of $bdf$.  Hence $\frac{bd}{uv} f=g_1 h_1$, where $g_1,h_1$ are primitive.  Since $R$ has the GL-property, $\frac{bd}{uv} f$ is primitive, so $\frac{bd}{uv}$ is a unit.  But then $f$ is reducible over $R$, a contradiction.  

a) $\Longrightarrow$ b). Assume a) holds.  Suppose $f=gh$ in $R[x]$ and the greatest common divisor of the coefficients of $f$ exists, say $s$.  We can assume $\deg f \geqslant 1$. Then $f=sf^{\prime}$, where $f^{\prime}$ is primitive.  Since $f^{\prime}$ is primitive, $f^{\prime}$ is the product of irreducibles, say $f^{\prime}=f_1f_2 \cdots f_k$.  By unique factorization in $K[x]$, $g=uf_1f_2 \cdots f_r$ for some $r \leqslant k$ (without loss of generality) and some $u \in K$. Since $R$ is PSP, $u \in R$ and the greatest common divisor of the coefficients of $g$ equals $u$. 
\end{proof}

In the paper \cite{AS}, Arnold and Sheldon proved the following theorem.

\begin{thm}
Let $R$ be a domain with quotient field $K$. The following are equivalent.
\begin{itemize}
\item[1)] $R[x]$ is an AP-Domain
\item[2)] $R[x]$ is a GL-Domain
\item[3)] Each of the following holds:
\begin{itemize}
\item[($\alpha$)] $R$ has the PSP-property
\item[($\beta$)] $R$ is integrally closed, and
\item[($\gamma$)] Whenever $B,C$ are finitely generated fractional ideals of $R$ such that $(BC)_v=R$, then $B_v$ is principal
\end{itemize}
\end{itemize}
\end{thm}
  
Condition $(\gamma)$ clearly has a resemblance to condition $b)$ of the proposition we just proved.  In fact, we have the following theorem.

\begin{prop}
Let $R$ be an integrally closed PSP-domain. The following are equivalent.
\begin{itemize}
\item[a)] Every nonconstant irreducible polynomial $f \in R[x]$ is irreducible in $K[x]$
\item[b)] Whenever $f=gh$ in $R[x]$ and the greatest common divisor of the coefficients of $f$ exists, then the greatest common divisor of the coefficients of $g$ exists
\item[c)] Whenever $B,C$ are finitely generated fractional ideals of $R$ such that $(BC)_v=R$, then $B_v$ is principal
\end{itemize}
\end{prop}

\begin{proof}  We already proved the equivalence of a) and b).

b) $\Longrightarrow$ c).  Suppose $(BC)_v=R$.  Let $g$ be a polynomial whose coefficients are the generators of $B$ (if generators of $B$ are chosen in order $B=(b_0,b_1,\cdots,b_n)$ we will define, $g=b_0+b_1x+\cdots +b_nx^n$) and let $h$ be defined similarly with respect to chosen generators of $C$.  Choose nonzero $b,c \in R$ such that $bg,ch \in R[x]$.  Since $R$ is integrally closed, $(A_{gh})_v=(A_gA_h)_v$ \cite[Proposition 34.8]{G}, so that $(A_{gh})_v=R$. Since  $(bcA_{gh})_v=bc((A_{gh})_v)$ \cite[Proposition 32.1(1)]{G}, we therefore have $(bcA_{gh})_v=bcR$.  This implies that the greatest common divisor of the coefficients of $bcgh$ equals $bc$. Since $bcgh=(bg)(ch)$, it follows by assumption that the greatest common divisor of the coefficients of $bg$ exists, so that $(A_{bg})_v=(bB)_v$ is principal \cite[Theorem 3.3]{AQ}.  Hence $B_v$ is principal.

c) $\Longrightarrow$ b).  Suppose $f=gh$ in $R[x]$ and the greatest common divisor of the coefficients of $f$ exists, say $s$.  Let $f_1=\frac{f}{s}$ and $h_1=\frac{h}{s}$.  Then $(A_gA_{h_1})_v=(A_{gh_1})_v=(A_{f_1})_v$, so $(A_gA_{h_1})_v=R$.  Hence $(A_g)_v$ is principal.  Hence the greatest common divisor of the coefficients of $g$ exists.  
\end{proof}  
  
Putting together the results of this section we obtain our main result, which is the following.

\begin{thm}
Let $R$ be a domain with quotient field $K$. The following are equivalent.
\begin{itemize}
\item[1)] Every nonconstant irreducible polynomial $f \in R[x]$ is irreducible in $K[x]$
\item[2)] $R[x]$ is an AP-Domain
\item[3)] $R[x]$ is a GL-Domain
\item[4)] Each of the following holds:
\begin{itemize}
\item[($\alpha$)] $R$ has the PSP-property
\item[($\beta$)] $R$ is integrally closed, and
\item[($\gamma$)] Whenever $B,C$ are finitely generated fractional ideals of $R$ such that $(BC)_v=R$, then $B_v$ is principal
\end{itemize}
\end{itemize}
\end{thm}

We close with a few observations. First, it is of note that if $R[x]$  is atomic, then $R[x]$ is an AP-domain if and only if $R$ is a UFD. Also we note that our main theorem of the previous section has its resolution in this stronger result. Indeed, if  we have the hypothesis of Theorem \ref{semi}, then $R[x]$ is an AP-domain. Hence $R$ is an AP-domain. If $R$ has at least one atom then $\rho(R)=\rho(R[x])=1$.

Finally, if $R$ is a Pr\"{u}fer domain satisfying property (P) then since every finitely generated ideal is invertible we must have $B_v$ principal for each $B$. Hence $R$ is a GCD-domain. And if $[a_0,\cdots, a_n]=1$, then $(a_0,\cdots,a_n)^{-1}=R$.  Hence there exist $r_0,\cdots,r_n \in R$ such that $r_0a_0+\cdots+r_na_n=1$.  We conclude that the greatest common divisor of a finite set of elements is a linear combination of that set, and so $R$ is a Bezout domain.  Thus we obtain \cite[Theorem 28.8]{G}, which says (paraphrasing) that a Pr\"{u}fer domain $R$ has property (P) iff $R$ is Bezout.

\bibliographystyle{alpha}

\end{document}